\documentclass[11pt,reqno]{amsart}
\usepackage{CJK}
\usepackage{amssymb}
\usepackage{amsmath}
\usepackage{amsthm}
\usepackage{bm}

\usepackage{graphicx}

\newtheorem{Prop}{Proposition}[section]
\newtheorem{Thm}[Prop]{Theorem}
\newtheorem{Lem}[Prop]{Lemma}
\newtheorem{Cor}[Prop]{Corollary}

\newtheorem*{Thm0}{Theorem}
\newtheorem*{Cor0}{Corollary}

\theoremstyle{definition}
\newtheorem{Ex}[Prop]{Example}
\newtheorem{Def}[Prop]{Definition}

\begin{document}

\title{A NOTE ON THE CONCORDANCE INVARIANTS UPSILON AND PHI}

\author{Shida Wang}

\email{wang217@indiana.edu}

\begin{abstract}

Dai, Hom, Stoffregen and Truong defined a family of concordance invariants $\varphi_j$.
The example of a knot with zero Upsilon invariant but nonzero epsilon invariant previously given by Hom also has nonzero phi invariant.
We show there are infinitely many such knots that are linearly independent in the smooth concordance group.
In the opposite direction, we build infinite families of linearly independent knots with zero phi invariant but nonzero Upsilon invariant.
We also give a recursive formula for the phi invariant of torus knots.

\end{abstract}

\maketitle

\section{Introduction}

For any knot $K$, there are two recently defined invariants derived from knot Heegaard Floer theory:
the Upsilon invariant, defined by Ozsv\'{a}th-Stipsicz-Szab\'{o} in~\cite{upsilon}, is a piecewise linear function on~$[0,2]$ and denoted by~$\Upsilon_K(t)$;
the phi invariant, defined by Dai-Hom-Stoffregen-Truong in~\cite{phi}, is a sequence of integers and denoted by~$(\varphi_j(K))_{j=1}^\infty$.
These two invariants give homomorphisms from the smooth concordance group~$\mathcal{C}$
to the abelian groups of continuous functions on $[0,2]$ and of sequences of integers (both with pointwise addition as the group operation), respectively.
The invariants have shown their power by proving the following result about the subgroup~$\mathcal{C}_{TS}$ of~$\mathcal{C}$ consisting of topologically slice knots.

\begin{Thm0}\emph{(\cite[Theorem 1.20]{upsilon} and \cite[Theorem 1.12]{phi})} The group $\mathcal{C}_{TS}$ contains a direct summand isomorphic to~$\mathbb{Z}^\infty$.\end{Thm0}

Here~$\mathbb{Z}^\infty$ denotes the additive abelian group of infinite sequences of integers with at most finitely many nonzero terms.
One may wonder whether one of the two invariants is stronger than the other. In fact, Dai, Hom, Stoffregen and Truong show that the~$\varphi$-invariant is not weaker than the~$\Upsilon$-invariant.

\begin{Thm0}\emph{(\cite[Proposition 1.10]{phi})} There exists a knot with vanishing~$\Upsilon$-invariant but nonvanishing~$\varphi$-invariant.\end{Thm0}

The example in the above theorem is exactly the one with vanishing~$\Upsilon$-invariant but\linebreak nonvanishing~$\varepsilon$-invariant given by Hom in~\cite{example0},
where~$\varepsilon$ is another smooth concordance invariant defined earlier~\cite{epsilon}.
Recently the author gives infinite families of linearly independent knots with vanishing~$\Upsilon$-invariant but nonvanishing~$\varepsilon$-invariant~\cite{example1}.
In this note, we will partially compute the~$\varphi$-invariant of knots in these families and verify the following fact, which recovers the main result of~\cite{example1}.

\begin{Prop}\label{trivial}There exists a direct summand of $\mathcal{C}$ isomorphic to~$\mathbb{Z}^\infty$ such that each of its nonzero elements has vanishing~$\Upsilon$-invariant but nonvanishing $\varphi$-invariant.\end{Prop}

A very simple alternative proof of the main theorem and a conjecture in~\cite{secondary1} will also be provided in the course of the above computation of the~$\varphi$-invariant. (The conjecture was first proved by Xu~\cite{secondary2}.)

On the contrary, \cite[Proposition 1.10]{phi} gives a knot Floer complex, a knot with which would have vanishing~$\varphi$-invariant but nonvanishing~$\Upsilon$-invariant.
It is not known yet whether this complex can be realized by a knot,
but there are many other knots that also have vanishing~$\varphi$-invariant and nonvanishing~$\Upsilon$-invariant.
The reason is that the~$\varphi$-invariant of any knot is a linear combination of that of torus knots~\cite[Example~1.7]{phi} while this is not true for the~$\Upsilon$-invariant~\cite{semigroup}.
With this observation, we actually can prove a stronger statement.

\begin{Prop}\label{easy}There exists a direct summand of $\mathcal{C}$ isomorphic to $\mathbb{Z}^\infty$ such that each of its nonzero elements has vanishing~$\varphi$-invariant but nonvanishing $\Upsilon$-invariant.\end{Prop}

We will also give a second proof of this proposition, using different knots.
The proof is based on a formula for the~$\varphi$-invariant of torus knots.
Recall that Feller and Krcatovich proved the following recursive formula.

\begin{Thm}\label{expansion}\emph{(\cite[Proposition 2.2]{recursive})} Suppose $p$ and $r$ are relatively prime positive integers and $k$ is a nonnegative integer.
Then $\Upsilon_{T_{p,kp+r}}(t)=\Upsilon_{T_{r,p}}(t)+k\Upsilon_{T_{p,p+1}}(t)$.\end{Thm}

From this formula, one immediately knows that any knot of the form $T_{p,kp+r}\#-T_{r,p}\#-kT_{p,p+1}$ has vanishing~$\Upsilon$-invariant,
where~$-K$ means the mirror image of the knot~$K$ with reversed orientation (representing the inverse element of~$K$ in~$\mathcal{C}$) and~$kK$ means the connected sum of~$k$ copies of~$K$.
In light of the fact that there are many known examples of knots with vanishing~$\Upsilon$-invariant, Proposition~\ref{trivial} is quite natural.

As a counterpart, we will provide a recursive formula for the $\varphi$-invariant of torus knots.

\begin{Thm}\label{main}Suppose $p$ and $r$ are relatively prime positive integers with~$r<p$ and~$k$ is a nonnegative integer.
Then $\varphi(T_{p,kp+r})=(k+1)\varphi(T_{r,p})+k\varphi(T_{p-r,p})+k(\varphi(T_{p,p+1})-\varphi(T_{p-1,p}))$.\end{Thm}

The application of this result includes producing more families satisfying Proposition~\ref{easy} and the following consequence.

\begin{Cor0}There exist knots with vanishing~$\varphi$-invariant but arbitrarily large splitting\linebreak concordance genus.\end{Cor0}


\emph{Acknowledgments.} The author wishes to express sincere thanks to Professor Robert Lipshitz for carefully reading a draft of this paper
and detailed suggestions on grammar.
Thanks also to Professors Charles Livingston and Jennifer Hom for comments.

\section{Notations and Conventions}

\subsection{Preliminaries}

We assume the reader is familiar with knot Floer homology, defined by Ozsv\'{a}th-Szab\'{o}~\cite{CFKdef1} and independently Rasmussen~\cite{CFKdef2},
the~$\varepsilon$-invariant, defined by Hom~\cite{epsilon},
the~$\Upsilon$-invariant, defined by Ozsv\'{a}th-Stipsicz-Szab\'{o}~\cite{upsilon},
and the~$\varphi$-invariant, defined by Dai-Hom-Stoffregen-Truong in~\cite{phi}.
We briefly recall some properties for later use.

To a knot $K\subset S^3$, the knot Floer complex associates a doubly filtered, free, finitely generated chain complex over~$\mathbb{F}[U,U^{-1}]$, denoted by~$\mathit{CFK}^\infty(K)$,
where~$\mathbb{F}$ is the field with two elements.
Up to filtered chain homotopy equivalence, this complex is an invariant of~$K$.
The connected sum operation corresponds to the tensor product operation, and the negative of a knot\linebreak corresponds to the dual of a complex.
There is a class $\mathfrak{C}$ of doubly filtered, free, finitely generated chain complexes over~$\mathbb{F}[U,U^{-1}]$ satisfying certain homological conditions (see~\cite[Definition~2.2]{1nn1} for details).
The class $\mathfrak{C}$ includes~$\mathit{CFK}^\infty(K)$ for all $K$, and there are complexes in~$\mathfrak{C}$ that cannot be realized by knots.
For any complex $C$ in~$\mathfrak{C}$, the number~$\varepsilon(C)$ in~$\{-1,0,1\}$ is an invariant of $C$ up to filtered chain homotopy~\cite{epsilon}.

An equivalence relation $\sim$ on~$\mathfrak{C}$ respecting the tensor product and dual operations, called \emph{stable equivalence}, is defined in~\cite{survey},
which satisfies that~$\mathit{CFK}^\infty(K)\sim\mathit{CFK}^\infty(K')$ for any smoothly concordant knots~$K$ and~$K'$.
Another equivalence relation~$\sim_\varepsilon$ that also satisfies this property is defined by~$C\sim_\varepsilon K'\Leftrightarrow\varepsilon(C\otimes C'^*)=0$,
where~$C'^*$ denotes the dual of~$C'$.
The relation~$\sim_\varepsilon$ is called \emph{$\varepsilon$-equivalence}.
It is coarser than~$\sim$, that is,~$C\sim C'\Rightarrow C\sim_\varepsilon C'$ for any~$C,C'\in\mathfrak{C}$.
Hence~$\mathfrak{C}/\sim$ and~$\mathfrak{C}/\sim_\varepsilon$ are both abelian groups and there is an obvious quotient homomorphism~$\mathfrak{C}/\sim\rightarrow\mathfrak{C}/\sim_\varepsilon$.

For any~$C\in\mathfrak{C}$, the~$\varepsilon$-equivalence class of~$C$ is denoted by~$[C]$.
For any knot~$K$, the~$\varepsilon$-equivalence class of~$K$ means~$[\mathit{CFK}^\infty(K)]$ and is sometimes denoted by~$[\![K]\!]$.
The group~$\mathfrak{C}/\sim_\varepsilon$ is denoted by~$\mathcal{CFK}_\mathrm{alg}$ in~\cite{1nn1}.
It has a subgroup~$\mathcal{CFK}:=\{[\![K]\!]\mid K\text{ is a knot}\}$, which is a quotient group of the smooth concordance group~$\mathcal{C}$.

For any complex~$C$ in~$\mathfrak{C}$, the invariants~$\Upsilon(C)$ and~$\varphi(C)=(\varphi_j(C))_{j=1}^\infty$
are a piecewise linear function on $[0,2]$ and an infinite sequence of integers with at most finitely many nonzero terms, respectively.
They are both invariant under stable equivalence and give homomorphisms from~$\mathfrak{C}/\sim$.
The invariant~$\varphi$ is further preserved by~$\sim_\varepsilon$ and hence factors through~$\mathfrak{C}/\sim_\varepsilon$, but this is not true for~$\Upsilon$.

For any knot~$K$, the invariants~$\varepsilon(\mathit{CFK}^\infty(K))$,~$\Upsilon(\mathit{CFK}^\infty(K))$ and~$\varphi(\mathit{CFK}^\infty(K))$
are abbreviated to~$\varepsilon(K)$,~$\Upsilon(K)$ and~$\varphi(K)$, respectively.
They are all computable once~$\mathit{CFK}^\infty(K)$ is known and therefore are computable for~$L$-space knots.

\subsection{The staircase complex and formal semigroups of $L$-space knots}

There is a special type of complexes in $\mathfrak{C}$, called \emph{staircase complexes},
that can be described by lengths of differential arrows (see~\cite[Section~2.4]{filtration} for example).
A staircase complex is usually encoded by an even number of positive integers~$b_1,\cdots,b_{2m}$ that are palindromic,\linebreak meaning~$b_i=b_{2m+1-i}$ for~$i=1,\cdots,2m$.
We will denote such a complex by~$\mathrm{St}(b_1,\cdots,b_{2m})$.
It corresponds to the complex~$C(b_1,-b_2,\cdots,b_{2m-1},-b_{2m})$ in~\cite{phi}.

An example is~$\mathit{CFK}^\infty(K)$ for any~$L$-space knot~$K$,
whose Alexander polynomial~$\Delta_K(t)$ must be of the form~$\sum_{i=0}^{2m}(-1)^it^{\alpha_i}$ with~$0=\alpha_0<\alpha_1<\cdots<\alpha_{2m}$~\cite[Theorem~1.2]{Lknot}\linebreak
and~$\frac{\alpha_{2m}}{2}=g(K)$ is the genus of~$K$~\cite[Theorem~1.2]{genus}.
The exponents of nonzero terms in~$\Delta_K(t)$ determine~$\mathit{CFK}^\infty(K)=\mathrm{St}(b_1,\cdots,b_{2m})$
by$$b_i=b_{2m+1-i}=\alpha_i-\alpha_{i-1},\forall i=1,\cdots,2m$$ (see~\cite[Remark~6.6]{ordered} for example).
Alternatively, one can define the \emph{formal\linebreak semigroup}~$S_K\subset\mathbb{Z}_{\geqslant0}$ of an~$L$-space knot~$K$ by the equation~$(1-t)(\sum_{s\in S}t^s)=\Delta_K(t)$.
Then \begin{equation}\label{alphaS}S_K=\{\alpha_0,\cdots,\alpha_1-1,\alpha_2,\cdots,\alpha_3-1,\cdots,\alpha_{2m-2},\cdots,\alpha_{2m-1}-1,\alpha_{2m}\}\cup\mathbb{Z}_{>\alpha_{2m}}\end{equation}
and~$b_1,\cdots,b_{2m}$ are determined by~$S_K$ in the following way:
\begin{equation}\label{bS}
\begin{split}
0,\cdots,b_1-1&\in S_K,\\
b_1,\cdots,b_1+b_2-1&\not\in S_K,\\
b_1+b_2,\cdots,b_1+b_2+b_3-1&\in S_K,\\
b_1+b_2+b_3,\cdots,b_1+b_2+b_3+b_4-1&\not\in S_K,\\
&\ \vdots\\
n&\in S_K, \forall n\geqslant b_1+\cdots+b_{2m}.
\end{split}
\end{equation}

Any torus knot~$T_{p,q}$ is an~$L$-space knot, where~$p$ and~$q$ are relatively prime positive integers.
Its formal semigroup~$S_{T_{p,q}}=\langle p,q\rangle=\{px+qy\mid x,y\in\mathbb{Z}_{\geqslant0}\}$
equals the analytically defined semigroup~\cite{Puiseux1} of the link of the singular point~$(0,0)$ on the plane\linebreak curve~$\{(z,w)\in\mathbb{C}^2\mid z^q-w^p=0\}$.

\subsection{The computation of $\varphi$ for $L$-space knots}

For $L$-space knots, we see in the above subsection that the Alexander polynomial, the knot Floer complex and the formal semigroup determine one another.
\cite[Proposition~1.5]{phi} gives the formula to compute~$\varphi$ for $L$-space knots from the Alexander polynomial.
It can also be expressed in terms of the formal semigroup.


We first introduce the following terminology for later convenience.

\begin{Def}Let $A$ be a subset of~$\mathbb{Z}$ and~$j$ be a positive integer.
A set~$B$ consisting of~$j$ consecutive integers is called a~\emph{$j$-gap of~$A$} if~$B\cap A=\emptyset$ but~$\{\min B-1,\max B+1\}\subset A$.
We use $\Phi_j(A)$ to denote the number of~$j$-gaps of~$A$ if this number is finite and use~$\Phi(A)$ to denote the sequence~$(\Phi_j(A))_{j=1}^\infty$ if each term is defined.\end{Def}

\begin{Ex}The sets $\{3,4,5,6\}$ and $\{9,10,11,12\}$ are two $4$-gaps of the\linebreak set $A=\{0,2,7,8,13\}\cup\mathbb{Z}_{>13}$.\\
We have $\Phi_1(A)=1$, $\Phi_4(A)=2$ and~$\Phi_j(A)=0,\forall j\neq1,4$.
Hence~$\Phi(A)=(1,0,0,1,0,0,0,0,\cdots)$.\end{Ex}

It is easy to verify that the function~$\Phi$ has following properties. The proof is omitted.

\begin{Lem}\label{gap}Suppose $A$ is a subset of~$\mathbb{Z}$.
\begin{enumerate}
\item If an integer $a\in A$, then~$\Phi(A)=\Phi(A\cap(-\infty,a])+\Phi(A\cap[a,\infty))$;
\item More generally, if integers~$a_1,\cdots,a_n\in A$ and~$a_1<\cdots<a_n$, then
$$\Phi(A)\linebreak=\Phi(A\cap(-\infty,a_1])+\Phi(A\cap[a_1,a_2])+\cdots+\Phi(A\cap[a_{n-1},a_n])+\Phi(A\cap[a_n,\infty));$$
\item For any $b\in\mathbb{Z}$, we have~$\Phi(\{b\})=\Phi(\mathbb{Z}_{\geqslant b})=\Phi(\mathbb{Z}_{\leqslant b})=(0,0,0,\cdots)$;
\item For any $b\in\mathbb{Z}$, we have~$\Phi(b+A)=\Phi(A)$, where~$b+A:=\{b+a\mid a\in A\}$; and
\item For any $b\in\mathbb{Z}$, we have~$\Phi(b-A)=\Phi(A)$, where~$b-A:=\{b-a\mid a\in A\}$.
\end{enumerate}\end{Lem}

The motivation to define the function~$\Phi$ is as follows.

\begin{Lem}\label{phiPhi}For any $L$-space knot $K$ with formal semigroup~$S$, we have~$\varphi(K)=\Phi(S)$.\end{Lem}

\begin{proof}[\emph{\bfseries Proof.}]By~\cite[Proposition~1.5]{phi},
for an $L$-space knot~$K$ with~$\Delta_K(t)=\sum_{i=0}^{2m}(-1)^it^{\alpha_i}$ and~$b_i=\alpha_i-\alpha_{i-1},\forall i=1,\cdots,2m$,
we have~$\varphi_j(K)=\#\{b_{2i}\mid i\in\{1,\cdots,m\}\text{ and }b_{2i}=j\}$ for any positive integer~$j$.
The conclusion follows from Equation~\eqref{alphaS} immediately.
\end{proof}

Now it is easy to obtain the following estimate of the~$\varphi$-invariant for torus knots, which can be viewed as a simplified version of Theorem~\ref{main}.

\begin{Prop}\label{rough}Suppose $p$ and $r$ are relatively prime positive integers with $r<p$ and $k$ is a positive integer. Then
\begin{enumerate}
\item $\varphi_j(T_{p,kp+r})=0$ for $j\geqslant p$;
\item $\varphi_j(T_{p,kp+r})=0$ for $\max\{r,p-r\}\leqslant j\leqslant p-2$ if~$2\leqslant r\leqslant p-2$; and
\item $\varphi_{p-1}(T_{p,kp+r})=k$.
\end{enumerate}\end{Prop}

\begin{proof}[\emph{\bfseries Proof.}]The semigroup of the torus knot $T_{p,kp+r}$ is $$S=\langle p,kp+r\rangle=\{0,p,2p,\cdots,kp,kp+r,kp+p,\cdots\}.$$
Note that there cannot be~$p$ consecutive numbers in~$\mathbb{Z}_{\geqslant0}\setminus S$, since~$p$ is a generator of~$S$.
Hence~$\Phi_j(S)=0$ for~$j\geqslant p$ and (i) follows.

Because $p,kp+r\in S$ and $S$ is a semigroup, we know~$lp+r\in S$ for any~$l\geqslant k$.
Thus there are no $j$-gaps of $S\cap[kp,\infty)$ for $j\geqslant\max\{r,p-r\}$.
Clearly $S\cap(-\infty,kp]=\{0,p,2p,\cdots,kp\}$ has~$k$ $(p-1)$-gaps but no other gaps.
This shows (ii) and (iii) by Lemma~\ref{gap}~(i).
\end{proof}


\section{Some Knots in the Kernel of $\Upsilon$}

By Theorem~\ref{expansion}, $\Upsilon_{T_{p,p+r}}(t)=\Upsilon_{T_{r,p}\#T_{p,p+1}}(t)$ for any relatively prime positive integers~$p$ and~$r$.
Allen proves that~$\mathit{CFK}^\infty(T_{p,p+2})\not\sim\mathit{CFK}^\infty(T_{2,p}\#T_{p,p+1})$ for any odd integer~$p\geqslant5$~\cite[Theorem 1.2]{secondary1}
and conjectures that~$\mathit{CFK}^\infty(T_{p,p+r})\not\sim\mathit{CFK}^\infty(T_{r,p}\#T_{p,p+1})$
for any relatively prime integers $p$ and $r$ with~$p\geqslant5$ and~$2\leqslant r\leqslant p-2$~\cite[Conjecture 5.3]{secondary1}.
This conjecture is proved by Xu~\cite[Theorem 1.2]{secondary2}.
Using Proposition~\ref{rough}, we can easily provide another proof.

\begin{Prop}\label{p-2}Suppose $p$ and $r$ are relatively prime positive integers\linebreak with~$2\leqslant r\leqslant p-2$ and~$k$ is a positive integer.
Then~$\varphi_{p-2}(T_{p,kp+r})=0$\linebreak and~$\varphi_{p-2}(T_{r,p}\#kT_{p,p+1})=k$.
In particular,~$T_{p,kp+r}$ and~$T_{r,p}\#kT_{p,p+1}$ have different~$\varphi$-invariant
and therefore~$\mathit{CFK}^\infty(T_{p,kp+r})\not\sim\mathit{CFK}^\infty(T_{r,p}\#kT_{p,p+1})$.\end{Prop}

\begin{proof}[\emph{\bfseries Proof.}]Proposition~\ref{rough}~(ii) gives $\varphi_{p-2}(T_{p,kp+r})=0$ since~$2\leqslant r\leqslant p-2$.

Proposition~\ref{rough}~(i) gives $\varphi_{p-2}(T_{r,p})=0$ since~$p-2\geqslant r$.
It is easy to verify that~$\varphi_{p-2}(T_{p,p+1})=1$ (see Lemma~\ref{basis}).
Thus~$\varphi_{p-2}(T_{r,p}\#kT_{p,p+1})=k$ since~$\varphi_{p-2}$ is a homomorphism.
\end{proof}

\begin{Cor}\label{family:phi}Let $K_i$ be the knot~$T_{p_i,k_ip_i+r_i}\#-T_{r_i,p_i}\#-k_iT_{p_i,p_i+1}$ for each positive\linebreak integer~$i$,
where~$p_i$ and~$r_i$ are relatively prime positive integers with~$2\leqslant r_i\leqslant p_i-2$\linebreak and~$p_i\leqslant p_{i+1}-2$, and~$k_i$ is any positive integer.
Then for each~$i$, we\linebreak have~$\varphi_{p_i-2}(K_i)=-k_i$ and $\varphi_{p_j-2}(K_i)=0,\forall j>i$.\end{Cor}

\begin{proof}[\emph{\bfseries Proof.}]Obviously ~$\varphi_{p_i-2}(K_i)=\varphi_{p_i-2}(T_{p_i,k_ip_i+r_i})-\varphi_{p_i-2}(T_{r_i,p_i}\#k_iT_{p_i,p_i+1})=0-k_i$ by\linebreak Proposition~\ref{p-2}.

For any $j>i$, the hypothesis implies that~$p_j-2\geqslant p_i$.
Hence$$\varphi_{p_j-2}(K_i)=\varphi_{p_j-2}(T_{p_i,k_ip_i+r_i})-\varphi_{p_j-2}(T_{r_i,p_i})-k_i\varphi_{p_j-2}(T_{p_i,p_i+1})=0-0-k\cdot0=0$$ by Proposition~\ref{rough}~(i).
\end{proof}

\begin{proof}[\emph{\bfseries Proof of Proposition~\ref{trivial}.}]
Any linear combination of the family~$\{K_i\}_{i=1}^\infty$ in Corollary~\ref{family:phi} has vanishing $\Upsilon$-invariant by Theorem~\ref{expansion}.
It follows from the conclusion of Corollary~\ref{family:phi} and a straightforward linear algebra argument that the $\varphi$-invariant does not vanish if the linear combination is nontrivial.
Taking $k_i=1$, we obtain a direct summand of $\mathcal{C}$ isomorphic to~$\mathbb{Z}^\infty$
by applying~\cite[Lemma 6.4]{upsilon} to the family of knots~$\{K_i\}_{i=1}^\infty$ and homomorphisms~$\{-\varphi_{p_i-2}\}_{i=1}^\infty$.
\end{proof}

Clearly Proposition~\ref{trivial} recovers the main result of~\cite{example1}, since~$\varphi$ factors through~$\mathfrak{C}/\sim_\varepsilon$.
Furthermore, all families in the proof of~\cite[Theorem~1.1]{example1} can be shown to have\linebreak nonvanishing~$\varphi$-invariant by Corollary~\ref{family:phi}.
In fact, those families\linebreak are~$\{T_{p_i,k_ip_i+r_i}\#-T_{r_i,p_i}\#-k_iT_{p_i,p_i+1}\}_{i=1}^\infty$,
where~$p_i$ and~$r_i$ are relatively prime positive integers with~$4\leqslant r_i<p_i/2$ and~$p_i\leqslant r_{i+1}$,
and~$k_i$ is any positive integer.
The conditions~$4\leqslant r_i<p_i/2$ and~$p_i\leqslant r_{i+1}$ obviously imply that~$2\leqslant r_i\leqslant p_i-2$ and~$p_i\leqslant p_{i+1}-2$.
Therefore Corollary~\ref{family:phi} applies.

\section{Some Knots in the Kernel of $\varphi$}

The semigroup of the torus knot~$T_{p,p+1}$ has exactly one $j$-gap for each $j=1,\cdots,p-1$ and no other gaps.
Hence the~$\varphi$-invariant of such knots is simple, as given in~\cite[Example~1.7]{phi}.

\begin{Lem}\label{basis}For any integer~$p\geqslant2$, we have $\varphi(T_{p,p+1})=(\overbrace{1,\cdots,1}^{p-1},0,0,\cdots)$.
In particular, the~$\varphi$-invariant of any knot is a linear combination of that of knots in~$\{T_{p,p+1}\}_{p=2}^\infty$.\end{Lem}

Recall that the~$\Upsilon$-invariant~$\Upsilon_K(t)$ of any knot~$K$ is a piecewise linear function on~$[0,2]$.
Hence for any~$t\in(0,2)$, the number~$\Delta\Upsilon'_K(t):=\lim\limits_{t\rightarrow t+}\Upsilon_K'(t)-\lim\limits_{t\rightarrow t-}\Upsilon_K'(t)$ is well defined
and therefore gives a homomorphism from~$\mathcal{C}$ to~$\mathbb{R}$.
For any knot~$K$, the vanishing of the function $\Upsilon_K(t)$ clearly implies that of the numbers~$\{\Delta\Upsilon'_K(t)\}_{t\in(0,2)}$.
(The converse is also true,\linebreak since~$\Upsilon_K(0)=0$~\cite[Proposition~1.5]{upsilon}.)

\begin{proof}[\emph{\bfseries Proof of Proposition~\ref{easy}.}]
Consider the family of knots~$\{J_k\}_{k=3}^\infty$ in~\cite[Section~3.3]{semigroup}, where~$J_k=(T_{2,3})_{k,2k-1}$ is the~$(k,2k-1)$-cable of the right-handed trefoil knot.
For each positive integer~$j$,
$$\lambda_j:K\mapsto\tfrac{1}{2j-1}\Delta\Upsilon'_K(\tfrac{2}{2j-1})-\tfrac{1}{2j-1}\Delta\Upsilon'_K(\tfrac{4}{2j-1})$$
is a homomorphism from $\mathcal{C}$,
and it takes on values in~$\mathbb{Z}$ by~\cite[Corollary~8.2]{denominator}.
The proof of~\cite[Theorem~3.8]{semigroup} shows that~$\lambda_k(J_k)=1$ for any integer $k\geqslant3$ and~$\lambda_j(J_k)=0,\forall j>k$.

By Lemma~\ref{basis}, for each knot~$J_k$, there exists a connected sum~$L_k$ of knots in~$\{T_{p,p+1}\}_{p=2}^\infty$ such that~$\varphi(J_k)=\varphi(L_k)$.
Each knot $L_k$ satisfies that~$\lambda_j(J_k)=0,\forall j\in\mathbb{Z}_{>0}$ by~\cite[Proposition~3.4]{semigroup}.
Hence the family of knots~$\{J_k\#-L_k\}_{k=3}^\infty$ has vanishing~$\varphi$-invariant and satisfies that~$\lambda_k(J_k\#-L_k)=1,\forall k\geqslant3$ and~$\lambda_j(J_k\#-L_k)=0,\forall j>k$.
This family gives a direct summand of $\mathcal{C}$ isomorphic to~$\mathbb{Z}^\infty$ by~\cite[Lemma 6.4]{upsilon}.
Any nontrivial linear combination of the family has nonzero image under some~$\lambda_j$ and therefore has nonvanishing~$\Upsilon$-invariant.
\end{proof}

\section{Proof of the Recursive Formula}

We first review some classical concepts.
Let $S\subset\mathbb{Z}_{\geqslant0}$ be a numerical semigroup, which is a semigroup containing 0 such that~$\mathbb{Z}_{\geqslant0}\setminus S$ is finite.
Suppose~$p\in S$ is a positive integer.
Define the integer\begin{equation*}\label{beta}\beta_i(S,p):=\min\{s\in S\mid s\equiv i\ \mathrm{mod}\ p\}\end{equation*}for each~$i=0,\cdots,p-1$.
The set\begin{equation*}\Omega(S,p):=\{\beta_0(S,p),\cdots,\beta_{p-1}(S,p)\}=\{s\in S\mid s-p\not\in S\}\end{equation*} is called the \emph{Ap\'{e}ry set of~$S$ with respect to~$p$}~\cite{Apery}.
We order the elements in~$\Omega(S,p)$ and denote them by~$\omega_0(S,p),\cdots,\omega_{p-1}(S,p)$,
that is,\begin{equation*}\label{omega}\Omega(S,p)=\{\omega_0(S,p),\cdots,\omega_{p-1}(S,p)\}\text{ and }\omega_0(S,p)<\cdots<\omega_{p-1}(S,p).\end{equation*}

For each~$i=1,\cdots,p-1$,
we denote\begin{equation}\label{Ai}A_i(S,p):=S\cap[\lfloor\tfrac{\omega_{i-1}(S,p)}{p}\rfloor p,\lfloor\tfrac{\omega_i(S,p)}{p}\rfloor p].\end{equation}
Then~$S=A_1(S,p)\cup\cdots\cup A_{p-1}(S,p)\cup\mathbb{Z}_{\geqslant\lfloor\frac{\omega_{p-1}(S,p)}{p}\rfloor p}$
Note that~$A_i(S,p)$ could be a singleton since $\lfloor\frac{\omega_{i-1}(S,p)}{p}\rfloor$ could equal~$\lfloor\frac{\omega_i(S,p)}{p}\rfloor$.

For each~$i=0,\cdots,p-1$,
we denote\begin{equation}\label{kappa}\kappa_i(S,p):=\min\{a\in \mathbb{Z}_{\geqslant0}\mid a\equiv\omega_i(S,p)\ \mathrm{mod}\ p\}=\omega_i(S,p)-\lfloor\tfrac{\omega_i(S,p)}{p}\rfloor p.\end{equation}
Note that~$\{\kappa_0(S,p),\cdots,\kappa_{p-1}(S,p)\}=\{0,\cdots,p-1\}$ but~$\kappa_0(S,p),\cdots,\kappa_{p-1}(S,p)$ are not necessarily increasing.

Since~$0=\beta_0(S,p)\in\Omega(S,p)$, we have\begin{equation}\label{omega0}\kappa_0(S,p)=\omega_0(S,p)=0.\end{equation}

\begin{Lem}\label{copies}For each~$i=1,\cdots,p-1$,
$$\Phi(A_i(S,p))=(\lfloor\tfrac{\omega_i(S,p)}{p}\rfloor-\lfloor\tfrac{\omega_{i-1}(S,p)}{p}\rfloor)\:\Phi(\{\kappa_0(S,p),\cdots,\kappa_{i-1}(S,p),p\}).$$\end{Lem}

\begin{proof}[\emph{\bfseries Proof.}]
We will show that
\begin{equation}\label{copy}\Phi(S\cap[lp,(l+1)p]))=\Phi(\{\kappa_0(S,p),\cdots,\kappa_{i-1}(S,p),p\})
\text{ if }\lfloor\tfrac{\omega_{i-1}(S,p)}{p}\rfloor\leqslant l<\lfloor\tfrac{\omega_i(S,p)}{p}\rfloor.\end{equation}

We claim that $S\cap[lp,(l+1)p]=lp+\{\kappa_0(S,p),\cdots,\kappa_{i-1}(S,p),p\}$ if~$\lfloor\frac{\omega_{i-1}(S,p)}{p}\rfloor\leqslant l<\lfloor\frac{\omega_i(S,p)}{p}\rfloor$.

In fact, any element belonging to the left-hand side can be expressed as~$lp+a$ for\linebreak some~$a\in\{0,\cdots,p\}=\{\kappa_0(S,p),\cdots,\kappa_{p-1}(S,p),p\}$.
If~$a=\kappa_j(S,p)$ with~$j\geqslant i$,
then\linebreak $lp+a\equiv\kappa_j(S,p)\equiv\omega_j(S,p)\ \mathrm{mod}\ p$
but~$lp+a<\lfloor\frac{\omega_i(S,p)}{p}\rfloor p+a\leqslant\lfloor\frac{\omega_j(S,p)}{p}\rfloor p+\kappa_j(S,p)=\omega_j(S,p)$,
which violates the definition of~$\Omega(S,p)$.
Thus~$a\in\{\kappa_0(S,p),\cdots,\kappa_{i-1}(S,p),p\}$.

Conversely, any element belonging to the right-hand side is either~$lp+p=(l+1)p$\linebreak
or~$lp+\kappa_j(S,p)=lp+\omega_j(S,p)-\lfloor\frac{\omega_j(S,p)}{p}\rfloor p=\omega_j(S,p)+(l-\lfloor\frac{\omega_j(S,p)}{p}\rfloor)p$ for some~$j\leqslant i-1$,
which must be in~$S$.
Clearly the right-hand side of the equality in the claim is included in~$[lp,(l+1)p]$, so it is included the left-hand side.
This validates the claim.

Hence Equation~\eqref{copy} is true by Lemma~\ref{gap}~(iv). Therefore we have
\begin{align*}&\Phi(A_i(S,p))\\
=&\Phi(S\cap[\lfloor\tfrac{\omega_{i-1}(S,p)}{p}\rfloor p,\lfloor\tfrac{\omega_i(S,p)}{p}\rfloor p])\\
=&\Phi(S\cap[\lfloor\tfrac{\omega_{i-1}(S,p)}{p}\rfloor p,(\lfloor\tfrac{\omega_{i-1}(S,p)}{p}\rfloor+1)p])
+\cdots+\Phi(S\cap[(\lfloor\tfrac{\omega_i(S,p)}{p}\rfloor-1)p,\lfloor\tfrac{\omega_i(S,p)}{p}\rfloor p])\\
=&\Phi(\lfloor\tfrac{\omega_{i-1}(S,p)}{p}\rfloor p+\{\kappa_0(S,p),\cdots,\kappa_{i-1}(S,p),p\})\\
&+\cdots\\
&+\Phi((\lfloor\tfrac{\omega_i(S,p)}{p}\rfloor-1)p+\{\kappa_0(S,p),\cdots,\kappa_{i-1}(S,p),p\})\\
=&\underbrace{\Phi(\{\kappa_0(S,p),\cdots,\kappa_{i-1}(S,p),p\})+\cdots+\Phi(\{\kappa_0(S,p),\cdots,\kappa_{i-1}(S,p),p\})}
_{\lfloor\frac{\omega_i(S,p)}{p}\rfloor-\lfloor\frac{\omega_{i-1}(S,p)}{p}\rfloor}\end{align*}
by Lemma~\ref{gap}~(ii),~(iii).
\end{proof}

When $S=\langle p,q\rangle$ where~$p$ and~$q$ are relatively prime positive integers (and~$q$ is not necessarily greater that~$p$),
we have\begin{equation}\label{torus omega}\omega_i(S,p)=iq,\forall i=0,\cdots,p-1.\end{equation}
In fact, for each~$i=0,\cdots,p-1$, the integer~$iq\in\Omega(S,p)$ since~$iq-p\not\in S$.\linebreak
Otherwise,~$iq-p=px+qy$ for some~$x,y\in\mathbb{Z}_{\geqslant0}$, which implies that~$i-y=pz$ for some~$z\in\mathbb{Z}_{>0}$ and is impossible.
Hence~$\{iq\mid i=0,\cdots,p-1\}\subset\Omega(S,p)$ and the two sets must be equal as each has~$p$ elements.
Therefore the increasing integers~$\omega_0(S,p),\cdots,\omega_{p-1}(S,p)$ are~$0q,\cdots,(p-1)q$, respectively.

\begin{Ex}(1) Consider~$S=\langle5,8\rangle=\{0,5,8,10,13,15,16,18,20,21,23,24,25,26\}\cup\mathbb{Z}_{\geqslant28}$ and take~$p=5$. We have
\begin{equation*}
\left\{\begin{aligned}&\omega_0(S,p)=0\\&\omega_1(S,p)=8\\&\omega_2(S,p)=16\\&\omega_3(S,p)=24\\&\omega_4(S,p)=32\end{aligned}\right.\text{ and thus }
\left\{\begin{aligned}&\kappa_0(S,p)=0\\&\kappa_1(S,p)=3\\&\kappa_2(S,p)=1\\&\kappa_3(S,p)=4\\&\kappa_4(S,p)=2\end{aligned}\right.\text{ and }
\left\{\begin{aligned}&A_1(S,p)=\{0,5\}\\&A_2(S,p)=\{5,8,10,13,15\}\\&A_3(S,p)=\{15,16,18,20\}\\&A_4(S,p)=\{20,21,23,24,25,26,28,29,30\}\end{aligned}\right..
\end{equation*}
They satisfy Lemma~\ref{copies}:
\begin{equation*}
\left\{\begin{aligned}&\Phi(A_1(S,p))=(0,0,0,1,0,0,\cdots)=(\lfloor\tfrac{8}{5}\rfloor-\lfloor\tfrac{0}{5}\rfloor)\Phi(\{0,5\})\\
&\Phi(A_2(S,p))=(2,2,0,0,0,\cdots)=(\lfloor\tfrac{16}{5}\rfloor-\lfloor\tfrac{8}{5}\rfloor)\Phi(\{0,3,5\})\\
&\Phi(A_3(S,p))=(2,0,0,0,\cdots)=(\lfloor\tfrac{24}{5}\rfloor-\lfloor\tfrac{16}{5}\rfloor)\Phi(\{0,3,1,5\})\\
&\Phi(A_4(S,p))=(2,0,0,0,\cdots)=(\lfloor\tfrac{32}{5}\rfloor-\lfloor\tfrac{24}{5}\rfloor)\Phi(\{0,3,1,4,5\})\end{aligned}\right..
\end{equation*}

(2) Consider~$S'=\langle3,5\rangle=\{0,3,5,6\}\cup\mathbb{Z}_{\geqslant8}$ and still take~$p=5$. We have
\begin{equation*}
\left\{\begin{aligned}&\omega_0(S',p)=0\\&\omega_1(S',p)=3\\&\omega_2(S',p)=6\\&\omega_3(S',p)=9\\&\omega_4(S',p)=12\end{aligned}\right.\text{ and thus }
\left\{\begin{aligned}&\kappa_0(S',p)=0\\&\kappa_1(S',p)=3\\&\kappa_2(S',p)=1\\&\kappa_3(S',p)=4\\&\kappa_4(S',p)=2\end{aligned}\right.\text{ and }
\left\{\begin{aligned}&A_1(S',p)=\{0\}\\&A_2(S',p)=\{0,3,5\}\\&A_3(S',p)=\{5\}\\&A_4(S',p)=\{5,6,8,9,10\}\end{aligned}\right..
\end{equation*}
They satisfy Lemma~\ref{copies}:
\begin{equation*}
\left\{\begin{aligned}&\Phi(A_1(S',p))=(0,0,0,\cdots)=(\lfloor\tfrac{3}{5}\rfloor-\lfloor\tfrac{0}{5}\rfloor)\Phi(\{0,5\})\\
&\Phi(A_2(S',p))=(1,1,0,0,0,\cdots)=(\lfloor\tfrac{6}{5}\rfloor-\lfloor\tfrac{3}{5}\rfloor)\Phi(\{0,3,5\})\\
&\Phi(A_3(S',p))=(0,0,0,\cdots)=(\lfloor\tfrac{9}{5}\rfloor-\lfloor\tfrac{6}{5}\rfloor)\Phi(\{0,3,1,5\})\\
&\Phi(A_4(S',p))=(1,0,0,0,\cdots)=(\lfloor\tfrac{12}{5}\rfloor-\lfloor\tfrac{9}{5}\rfloor)\Phi(\{0,3,1,4,5\})\end{aligned}\right..
\end{equation*}
\end{Ex}

\begin{proof}[\emph{\bfseries Proof of Theorem~\ref{main}.}]
Denote$$U:=S_{T_{p,kp+r}}=\langle p,kp+r\rangle,V:=S_{T_{r,p}}=\langle p,r\rangle\text{ and }W:=S_{T_{p-r,p}}=\langle p,p-r\rangle.$$
Because~$lp\in U,\forall l\in\mathbb{Z}_{\geqslant0}$, we have $\Phi(U)=\Phi(A_1(U,p))+\cdots+\Phi(A_{p-1}(U,p))$ by Lemma~\ref{gap}~(ii),~(iii).
Similarly $\Phi(V)=\Phi(A_1(V,p))+\cdots+\Phi(A_{p-1}(V,p))$ and $\Phi(W)=\Phi(A_1(W,p))+\cdots+\Phi(A_{p-1}(W,p))$.
We will show that $$\Phi(A_i(U,p))=(k+1)\Phi(A_i(V,p))+k\Phi(A_i(W,p))$$ for each~$i=2,\cdots,p-1$ and then consider~$\Phi(A_1(U,p))$,~$\Phi(A_1(V,p))$ and~$\Phi(A_1(W,p))$ separately.

Now fix $i\in\{2,\cdots,p-1\}$. (If~$p=2$, this set is empty and there is nothing to prove.)

Since~$\Phi(A_i(S,p))=(\lfloor\frac{\omega_i(S,p)}{p}\rfloor-\lfloor\frac{\omega_{i-1}(S,p)}{p}\rfloor)\:\Phi(\{\kappa_0(S,p),\cdots,\kappa_{i-1}(S,p),p\})$ for $S=U,V,W$ by Lemma~\ref{copies},
it suffices to prove that\begin{equation}\label{UV}\Phi(\{\kappa_0(U,p),\cdots,\kappa_{i-1}(U,p),p\})=\Phi(\{\kappa_0(V,p),\cdots,\kappa_{i-1}(V,p),p\}),\end{equation}
that \begin{equation}\label{VW}\Phi(\{\kappa_0(V,p),\cdots,\kappa_{i-1}(V,p),p\})=\Phi(\{\kappa_0(W,p),\cdots,\kappa_{i-1}(W,p),p\})\end{equation}
and that\begin{equation}\label{UVW}\lfloor\tfrac{\omega_i(U,p)}{p}\rfloor-\lfloor\tfrac{\omega_{i-1}(U,p)}{p}\rfloor
=(k+1)(\lfloor\tfrac{\omega_i(V,p)}{p}\rfloor-\lfloor\tfrac{\omega_{i-1}(V,p)}{p}\rfloor)+k(\lfloor\tfrac{\omega_i(W,p)}{p}\rfloor-\lfloor\tfrac{\omega_{i-1}(W,p)}{p}\rfloor).\end{equation}

Note that $\kappa_l(U,p)=l(kp+r)-\lfloor\frac{l(kp+r)}{p}\rfloor p$ for each~$l=0,\cdots,p-1$ by Equations~\eqref{kappa} and~\eqref{torus omega}.
Similarly, $\kappa_l(V,p)=lr-\lfloor\frac{lr}{p}\rfloor p$ and $\kappa_l(W,p)=l(p-r)-\lfloor\frac{l(p-r)}{p}\rfloor p$ for each~$l=0,\cdots,p-1$.

Clearly~$\kappa_l(U,p)=\kappa_l(V,p)$ for each~$l=0,\cdots,p-1$.
In paticular, $$\{\kappa_0(U,p),\cdots,\kappa_{i-1}(U,p)\}\cup\{p\}=\{\kappa_0(V,p),\cdots,\kappa_{i-1}(V,p)\}\cup\{p\},$$ which implies Equation~\eqref{UV}.

Observe that $p$ divides $\kappa_l(V,p)\!+\!\kappa_l(W,p)\!=\!lp\!-\!(\lfloor\frac{lr}{p}\rfloor\!+\!\lfloor\frac{l(p-r)}{p}\rfloor)p$ for each~$l\!=\!0,\!\cdots\!,p\!-\!1$.
Both~$\{\kappa_0(V,p),\cdots,\kappa_{p-1}(V,p)\}$ and~$\{\kappa_0(W,p),\cdots,\kappa_{p-1}(W,p)\}$ equal~$\{0,\cdots,p-1\}$ by definition.
Moreover,~$\kappa_0(V,p)=\kappa_0(W,p)=0$ by Equation~\eqref{omega0}.
Hence~$\kappa_l(W,p)=p-\kappa_l(V,p)$ for\linebreak each~$l=1,\cdots,p-1$.
Therefore $\{\kappa_0(W,p),\cdots,\kappa_{i-1}(W,p),p\}=p-\{\kappa_0(V,p),\cdots,\kappa_{i-1}(V,p),p\}$, which implies Equation~\eqref{VW} by Lemma~\ref{gap}~(v).

To prove Equation~\eqref{UVW}, we must show
$$\lfloor\tfrac{i(kp+r)}{p}\rfloor-\lfloor\tfrac{(i-1)(kp+r)}{p}\rfloor
=(k+1)(\lfloor\tfrac{ir}{p}\rfloor-\lfloor\tfrac{(i-1)r}{p}\rfloor)+k(\lfloor\tfrac{i(p-r)}{p}\rfloor-\lfloor\tfrac{(i-1)(p-r)}{p}\rfloor)$$
because of Equation~\eqref{torus omega}.
Recall that~$i\in\{2,\cdots,p-1\}$ and that ~$p$ and~$r$ are relatively prime.
We know~$ir$ and~$i(p-r)$ are two integers not divisible by~$p$ and thus$$\lfloor\tfrac{ir}{p}\rfloor+\lfloor\tfrac{i(p-r)}{p}\rfloor=\tfrac{ir}{p}+\tfrac{i(p-r)}{p}-1=i.$$
Similarly$$\lfloor\tfrac{(i-1)r}{p}\rfloor+\lfloor\tfrac{(i-1)(p-r)}{p}\rfloor=\tfrac{(i-1)r}{p}+\tfrac{(i-1)(p-r)}{p}-1=i-1.$$
Subtracting these identities, we obtain$$(\lfloor\tfrac{ir}{p}\rfloor-\lfloor\tfrac{(i-1)r}{p}\rfloor)+(\lfloor\tfrac{i(p-r)}{p}\rfloor-\lfloor\tfrac{(i-1)(p-r)}{p}\rfloor)=1.$$
Since~$r<p$, the number~$\lfloor\frac{ir}{p}\rfloor-\lfloor\frac{(i-1)r}{p}\rfloor$ is either~$1$ or~$0$.
Hence there are only two cases:\begin{equation*}
\left\{\begin{aligned}&\lfloor\tfrac{ir}{p}\rfloor-\lfloor\tfrac{(i-1)r}{p}\rfloor=1\\&\lfloor\tfrac{i(p-r)}{p}\rfloor-\lfloor\tfrac{(i-1)(p-r)}{p}\rfloor=0\end{aligned}\right.\text{ and }
\left\{\begin{aligned}&\lfloor\tfrac{ir}{p}\rfloor-\lfloor\tfrac{(i-1)r}{p}\rfloor=0\\&\lfloor\tfrac{i(p-r)}{p}\rfloor-\lfloor\tfrac{(i-1)(p-r)}{p}\rfloor=1\end{aligned}\right..
\end{equation*}
Therefore~$\lfloor\frac{i(kp+r)}{p}\rfloor-\lfloor\frac{(i-1)(kp+r)}{p}\rfloor=k+\lfloor\frac{ir}{p}\rfloor-\lfloor\frac{(i-1)r}{p}\rfloor$ must be equal\linebreak
to~$(k+1)(\lfloor\frac{ir}{p}\rfloor-\lfloor\frac{(i-1)r}{p}\rfloor)+k(\lfloor\frac{i(p-r)}{p}\rfloor-\lfloor\frac{(i-1)(p-r)}{p}\rfloor)$ no matter which case holds.
This completes the proof of Equation~\eqref{UVW}.

With Equations~\eqref{UV},~\eqref{VW} and~\eqref{UVW},
we verified~$\Phi(A_i(U,p))=(k+1)\Phi(A_i(V,p))+k\Phi(A_i(W,p))$ for each~$i=2,\cdots,p-1$.
To see what happens for~$i=1$, note that~$A_1(U,p)=\{0,p,\cdots,kp\}$, $A_1(V,p)=\{0\}$ and~$A_1(W,p)=\{0\}$ by Equations~\eqref{Ai} and~\eqref{torus omega}.
Clearly~$\Phi(A_1(U,p))$ is the infinite sequence with the~$(p-1)$th term~$k$ and all the other terms~$0$, which\linebreak equals~$k(\varphi(T_{p,p+1})-\varphi(T_{p-1,p}))$ by Lemma~\ref{basis}.
Hence$$\Phi(A_1(U,p))=(k+1)\Phi(A_1(V,p))+k\Phi(A_1(W,p))+k(\varphi(T_{p,p+1})-\varphi(T_{p-1,p})).$$
Therefore we have~$\Phi(U)=(k+1)\Phi(V)+k\Phi(W)+k(\varphi(T_{p,p+1})-\varphi(T_{p-1,p}))$, which concludes the theorem by Lemma~\ref{phiPhi}.
\end{proof}

\section{Applications of the Recursive Formula}

First, Theorem~\ref{main} allows us to compute the~$\varphi$-invariant of a torus knot without writing down its semigroup or Alexander polynomial.

\begin{Ex}(cf. \cite[Example~2.8]{recursive}) Consider the torus knot~$T_{8,11}$. We have
\begin{align*}
&\!\varphi(T_{8,1\cdot8+3})\\
=\quad&(1+1)\varphi(T_{3,8})+1\varphi(T_{8-3,8})+1(\varphi(T_{8,9})-\varphi(T_{7,8}))\\
=\quad&2\varphi(T_{3,8})+\varphi(T_{5,8})+(\varphi(T_{8,9})-\varphi(T_{7,8}))\\
=\quad&2(3\varphi(T_{2,3})+2\varphi(T_{1,3})+2(\varphi(T_{3,4})-\varphi(T_{2,3})))\\
+&(2\varphi(T_{3,5})+\varphi(T_{2,5})+(\varphi(T_{5,6})-\varphi(T_{4,5})))\\
+&(\varphi(T_{8,9})-\varphi(T_{7,8}))\\
=\quad&2(3\varphi(T_{2,3})+2(\varphi(T_{3,4})-\varphi(T_{2,3})))\\
+&(2(2\varphi(T_{2,3})+\varphi(T_{1,3})+(\varphi(T_{3,4})-\varphi(T_{2,3}))))\\
&+(3\varphi(T_{1,2})+2\varphi(T_{1,2})+2(\varphi(T_{2,3})-\varphi(T_{1,2})))\\
&+(\varphi(T_{5,6})-\varphi(T_{4,5})))\\
+&(\varphi(T_{8,9})-\varphi(T_{7,8}))\\
=\quad&2(3(1,0,0,0,\cdots)+2(0,1,0,0,0,\cdots))\\
+&(2(2(1,0,0,0,\cdots)+(0,1,0,0,0,\cdots))+2(1,0,0,0,\cdots)+(0,0,0,1,0,0,\cdots))\\
+&(0,0,0,0,0,0,1,0,0\cdots)\\
=\quad&2(3,2,0,0,0,\cdots)\\
+&(2(2,1,0,0,0,\cdots)+(2,0,0,0,\cdots)+(0,0,0,1,0,0,\cdots))\\
+&(0,0,0,0,0,0,1,0,0\cdots)\\
=\quad&(6,4,0,0,0,\cdots)+(6,2,0,1,0,0,\cdots)+(0,0,0,0,0,0,1,0,0\cdots)\\
=\quad&(12,6,0,1,0,0,1,0,0,0,0,\cdots)
\end{align*}
\end{Ex}

Next, we consider knots with vanishing~$\varphi$-invariant arising from Theorem~\ref{main}.

\begin{Lem}\label{phi0Upsilon}Suppose $p$ and $r$ are relatively prime positive integers with~$2\leqslant r\leqslant p-2$ and~$k$ is a nonnegative integer.
Let $K$ be the knot~$T_{p,kp+r}\#-(k+1)T_{r,p}\#-kT_{p-r,p}\#-k(T_{p,p+1}\#-T_{p-1,p})$.
Then~$\Delta\Upsilon'_K(\frac{2}{p-1})=k(p-1)$ and~$\Delta\Upsilon'_K(t)=0,\forall t\in(0,\frac{2}{p-1})$.\end{Lem}

\begin{proof}[\emph{\bfseries Proof.}]According to Theorem~\ref{expansion}, we have
\begin{align*}&\Upsilon_K(t)\\
=&\Upsilon_{T_{p,kp+r}}(t)-(k+1)\Upsilon_{T_{r,p}}(t)-k\Upsilon_{T_{p-r,p}}(t)-k(\Upsilon_{T_{p,p+1}}(t)-\Upsilon_{T_{p-1,p}}(t))\\
=&\Upsilon_{T_{r,p}}(t)+k\Upsilon_{T_{p,p+1}}(t)-(k+1)\Upsilon_{T_{r,p}}(t)-k\Upsilon_{T_{p-r,p}}(t)-k\Upsilon_{T_{p,p+1}}(t)+k\Upsilon_{T_{p-1,p}}(t)\\
=&k\Upsilon_{T_{p-1,p}}(t)-k\Upsilon_{T_{r,p}}(t)-k\Upsilon_{T_{p-r,p}}(t).\end{align*}
The location of the first singularity (the discontinuity of the derivative) of~$\Upsilon_{T_{r,p}}(t)$ is at~$\frac{2}{r-1}$ by~\cite[Theorem 8]{cobordism},
and that of~$\Upsilon_{T_{p-r,p}}(t)$ is at~$\frac{2}{p-r-1}$.
Both are greater than~$\frac{2}{p-1}$ by hypothesis.
Applying \cite[Proposition 6.3]{upsilon}, we know $\Delta\Upsilon'_{T_{p-1,p}}(\frac{2}{p-1})=p-1$\linebreak and~$\Delta\Upsilon'_{T_{p-1,p}}(t)=0,\forall t\in(0,\frac{2}{p-1})$.
This concludes the proof.
\end{proof}

Thus we immediately obtain a counterpart of Corollary~\ref{family:phi} with the roles of~$\Upsilon$ and~$\varphi$ interchanged.

\begin{Cor}\label{family:Upsilon}Let $K_i=T_{p_i,k_ip_i+r_i}\#-(k_i+1)T_{r_i,p_i}\#-k_iT_{p_i-r_i,p_i}\#-k_i(T_{p_i,p_i+1}\#-T_{p_i-1,p_i})$ for each positive integer~$i$,
where~$p_i$ and~$r_i$ are relatively prime positive integers\linebreak with~$2\leqslant r_i\leqslant p_i-2$ and~$p_i\leqslant p_{i+1}$, and~$k_i$ is any positive integer.
Then for each~$i$, we have~$\Delta\Upsilon'_{K_i}(\frac{2}{p_i-1})=k_i(p_i-1)$ and $\Delta\Upsilon'_{K_i}(\frac{2}{p_j-1})=0,\forall j>i$.\end{Cor}

With this Corollary, we can provide many families that satisfy the conclusion of Proposition~\ref{easy}.

\begin{proof}[\emph{\bfseries Alternative proof of Proposition~\ref{easy}.}]
Take any family of knots~$\{K_i\}_{i=1}^\infty$ in Corollary~\ref{family:Upsilon} with~$k_i=1,\forall i$.
For each positive integer~$i$,
$$\xi_i:K\mapsto\tfrac{1}{p_i-1}\Delta\Upsilon'_K(\tfrac{2}{p_i-1})$$
is a homomorphism from $\mathcal{C}$,
and it takes on values in~$\mathbb{Z}$ by~\cite[Proposition 1.7]{upsilon} or~\cite[Corollary~8.2]{denominator}.
Corollary~\ref{family:Upsilon} states that~$\xi_i(K_i)\!\!=\!\!1$ for any positive integer $i$\linebreak and~$\xi_k(K_i)\!=\!0,\forall k>i$.
Additionally, the family of knots~$\{K_i\}_{i=1}^\infty$ has vanishing~$\varphi$-invariant by Theorem~\ref{main}.
This family gives a direct summand of $\mathcal{C}$ isomorphic to~$\mathbb{Z}^\infty$ by~\cite[Lemma~6.4]{upsilon}.
Any nontrivial linear combination of the family has nonzero image under some~$\xi_i$ and therefore has nonvanishing~$\Upsilon$-invariant.
\end{proof}

Finally, we discuss bounds on the concordance genus.

Recall that the \emph{concordance genus}~$g_c$ of a knot~$K$ is defined to be$$\min\{g(K')\mid K\text{ is smoothly concordant to }K'\},$$
where~$g(K')$ denotes the genus of~$K'$.

A generalization is the \emph{splitting concordance genus}~$g_{\mathrm{sp}}$~\cite[Definition 1.1]{splitting} of a knot $K$,
defined as$$\min\{\max\{g(K_1),\cdots,g(K_m)\}\mid K\text{ is smoothly concordant to }K_1\#\cdots\#K_m\text{ for some }m\}.$$
Obviously~$g_{\mathrm{sp}}(K)\leqslant g_c(K)\leqslant g(K)$ for any knot~$K$.

Let
$$T(K):=\left\{\begin{aligned}&\frac{1}{\min\{t\in\mathbb{Q}\cap(0,2)\mid\Delta\Upsilon'_K(t)\neq0\}}&&\text{if }\Upsilon_K(t)\not\equiv0\\&0&&\text{if }\Upsilon_K(t)\equiv0\end{aligned}\right.$$
and
$$N(K):=\left\{\begin{aligned}&0&&\text{if }\varphi(K)=(0,0,0,0\cdots)\\&\max\{j\mid\varphi_j(K)\neq0\}&&\text{otherwise}\end{aligned}\right.$$
for any knot~$K$.
It is not hard to see that they are invariant under concordance and subadditive under connected sum.

A topological application of the~$\varphi$-invariant provided in~\cite[Theorem~1.14(1)]{phi} is\linebreak that~$g_c(K)\geqslant\frac{1}{2}N(K)$.
Thus~$g_{\mathrm{sp}}(K)$ is also bounded below by~$\frac{1}{2}N(K)$ by the subadditivity of~$N$ under connected sum.
On the other hand,~$T(K)$ is another lower bound of~$g_{\mathrm{sp}}(K)$ by~\cite[Proposition 4.2]{splitting}.
The knot $K$ in Lemma~\ref{phi0Upsilon} demonstrates an example for which the bound~$N(K)$ does not work but~$T(K)=\frac{p-1}{2}$.

\begin{Cor}There exist knots with vanishing~$\varphi$-invariant but arbitrarily large splitting concordance genus.\end{Cor}

\section{Further Remarks}

It is an open question whether there exists a knot with vanishing~$\varepsilon$-invariant but nonvanishing~$\Upsilon$-invariant.
Since the~$\varphi$-invariant factors through~$\varepsilon$-equivalence classes,
one may naturally wonder if the knot in Lemma~\ref{phi0Upsilon} provides such an example.
However, it does not.

\begin{Prop}\label{family:epsilon}The families of knots in Corollary~\ref{family:Upsilon} give subgroups of~$\mathcal{C}$ isomorphic to~$\mathbb{Z}^\infty$
such that each of its nonzero elements has vanishing~$\varphi$-invariant but nonvanishing~$\varepsilon$-invariant.\end{Prop}

To prove this statement, we adopt the same method and notations as those in~\cite{example1}.
We briefly review some for readers' convenience.

Recall that a \emph{totally ordered abelian group} is an abelian group with a total order respecting the addition operation.
For two elements $g,h\geqslant0$ of such a group, where~$0$ is the identity element, we write~$g\ll h$ if~$n\cdot g<h$ for any natural number~$n$.

The binary relation defined by~$[C]>[C']\Leftrightarrow\varepsilon(C\otimes C'^*)=1$ gives a total order\linebreak on~$\mathcal{CFK}_\mathrm{alg}:=\mathfrak{C}/\sim_\varepsilon$ (and hence on~$\mathcal{CFK}$)
that respects the addition operation~\cite[Proposition 4.1]{ordered}.

The $\varepsilon$-equivalence class of the complex~$\mathrm{St}(b_1,\cdots,b_{2m})$ is denoted by~$[b_1,\cdots,b_{2m}]$.
When some terms in the brackets are fixed, the other terms are allowed to be void.
For example, we could mean $[1,1]$ by writing~$[1,b_2,\cdots,b_{2m-1},1]$.

Now we summarize some needed facts.

\begin{Lem}\label{last1}The following statements are true.
\begin{enumerate}
\item\label{domination implies independence} \emph{(\cite[Lemma 4.7]{ordered})} If~$0<g_1\ll g_2\ll g_3\ll\cdots$ in a totally ordered abelian group,
then~$g_1,g_2,g_3,\cdots$ are linearly independent;
\item\label{a1greater} \emph{(\cite[Lemmas 6.3]{ordered})} if positive integers~$b_1$ and~$b_1'$ satisfy~$b_1>b_1'$,
then\linebreak $[b_1,b_2,\cdots,b_{2m-1},b_{2m}]\ll[b_1',b_2',\cdots,b_{2m'-1}',b_{2m'}']$ for any positive integers~$b_2,\cdots,b_{2m-1}$ and~$b_2',\cdots,b_{2m'-1}'$;
\item\label{a2greater} \emph{(\cite[Lemmas 6.4]{ordered})} if positive integers~$b_1$, $b_1'$, $b_2$ and $b_2'$ satisfy~$b_1=b_1'$\linebreak and~$b_2>b_2'$,
then~$[b_1,b_2,b_3,\cdots,b_{2m-2},b_{2m-1},b_{2m}]\gg[b_1',b_2',b_3',\cdots,b_{2m'-2}',b_{2m'-1}',b_{2m'}']$ for any positive integers~$b_3,\cdots,b_{2m-2}$ and~$b_3',\cdots,b_{2m'-2}'$; and
\item\label{split} \emph{(\cite[Lemma 3.1]{filtration} and~\cite[example~2.2]{example1})} If the positive integer $b_j\leqslant n$ for each~$j=1,\cdots,m$,
then $$[\underbrace{1,n,\cdots,1,n}_{k\text{ copies of }1,n},b_1,\cdots,b_m,b_m,\cdots,b_1,\underbrace{n,1,\cdots,n,1}_{k\text{ copies of }n,1}]=k[1,n,n,1]+[b_1,\cdots,b_m,b_m,\cdots,b_1].$$
\end{enumerate}\end{Lem}

The following bounds are enough for later use.

\begin{Lem}\label{last2}Suppose $p$ and $r$ are relatively prime positive integers with~$r<p$ and~$k$ is a positive integer.
\begin{enumerate}
\item\label{self bound} $[\![T_{p,kp+r}]\!]\ll[1,p,p,1]$ in $\mathcal{CFK}_\mathrm{alg}$.
\item\label{remainder} If $p\geqslant3$, then $[\![T_{p,kp+r}]\!]=k[1,p-1,p-1,1]+O$, where $0\leqslant O\ll[1,r,r,1]$ in $\mathcal{CFK}_\mathrm{alg}$.
\end{enumerate}\end{Lem}

\begin{proof}[\emph{\bfseries Proof.}]Assume $p\geqslant3$.
In the proof of Lemma~\ref{rough}~(i),
we have seen that the semigroup of~$T_{p,kp+r}$ has no~$j$-gaps for~$j\geqslant p$.
According to the relationship between the staircase complex and gaps shown in Equation~\eqref{bS} (and the palindromicity),
$$\mathit{CFK}^\infty(T_{p,kp+r})=\mathrm{St}(\underbrace{1,p-1,\cdots,1,p-1}_{k\text{ copies of }1,p-1},c_1,\cdots,c_{2l},\underbrace{p-1,1,\cdots,p-1,1}_{k\text{ copies of }p-1,1}),$$
where $c_1,\cdots,c_{2l}$ are positive integers less than~$p$.
Then$$[\![T_{p,kp+r}]\!]=k[1,p-1,p-1,1]+[c_1,\cdots,c_{2l}]$$ by Lemma~\ref{last1}~\ref{split}.
From the initial terms~$0,p,2p,\cdots,kp,kp+r,kp+p$ in~$\langle p,kp+r\rangle$, we know that either $c_1=2$ (when~$r=1$) or $(c_1,c_2)=(1,r-1)$ (when~$r>1$).
These two cases both imply~$[c_1,\cdots,c_{2l}]\ll[1,r,r,1]$ by Lemma~\ref{last1}~\ref{a1greater}~or~\ref{a2greater}.
Taking~$O=[c_1,\cdots,c_{2l}]$ (or~$O=0$ when~$l=0$), the second part of the lemma is proved.

To show the first part in the case of~$p\geqslant3$, note that~$[1,p-1,p-1,1]\ll[1,p,p,1]$ and that~$O\ll[1,r,r,1]\ll[1,p,p,1]$ by Lemma~\ref{last1}~\ref{a2greater}.
Then the conclusion follows from the definition of~$\ll$.
If~$p=2$, we can directly verify that~$[\![T_{p,kp+r}]\!]=[\underbrace{1,\cdots,1}_{2k}]\ll[1,2,2,1]$.
\end{proof}

\begin{Lem}\label{phi0epsilon}Suppose $p$ and $r$ are relatively prime positive integers with~$2\leqslant r\leqslant p-2$ and~$k$ is a positive integer.
Then $$[\![T_{p,kp+r}\#-(k+1)T_{r,p}\#-kT_{p-r,p}\#-k(T_{p,p+1}\#-T_{p-1,p})]\!]=k[1,p-2,p-2,1]\pm O,$$ where~$0\leqslant O\ll[1,p-2,p-2,1]$ in~$\mathcal{CFK}_\mathrm{alg}$.\end{Lem}

\begin{proof}[\emph{\bfseries Proof.}]Observe that
\begin{align*}&[\![T_{p,kp+r}\#-(k+1)T_{r,p}\#-kT_{p-r,p}\#-k(T_{p,p+1}\#-T_{p-1,p})]\!]\\
=&[\![T_{p,kp+r}]\!]-(k+1)[\![T_{r,p}]\!]-k[\![T_{p-r,p}]\!]-k[\![T_{p,p+1}]\!]+k[\![T_{p-1,p}]\!]\\
=&k([1,p-1,p-1,1]+O_1)-(k+1)O_2-kO_3\\
&-k([1,p-1,p-1,1]+O_4)+k([1,p-2,p-2,1]+O_5)\\
=&k[1,p-2,p-2,1]+kO_1-(k+1)O_2-kO_3-kO_4+kO_5.\end{align*}

Here $0\leqslant O_2\ll[1,r-1,r-1,1]$ and~$0\leqslant O_3\ll[1,p-r-1,p-r-1,1]$ by Lemma~\ref{last2}~\ref{self bound},
while~$0\leqslant O_1\ll[1,r,r,1]$, $0\leqslant O_4\ll[1,1,1,1]$ and $0\leqslant O_5\ll[1,1,1,1]$ by Lemma~\ref{last2}~\ref{remainder}.

Take $O=\pm(kO_1-(k+1)O_2-kO_3-kO_4+kO_5)$, whichever is nonnegative. The conclusion follows from the definition of $\ll$.
\end{proof}

\begin{proof}[\emph{\bfseries Proof of Proposition~\ref{family:epsilon}.}]
From Lemma~\ref{phi0epsilon}, we know~$[\![K_i]\!]=k_i[1,p_i-2,p_i-2,1]\pm O_i$ with~$0\leqslant O_i\ll[1,p_i-2,p_i-2,1]$.
This implies~$[\![K_i]\!]\ll[\![K_{i+1}]\!]$ for each~$i$, that\linebreak is,~$n[\![K_i]\!]<[\![K_{i+1}]\!]$ for any natural number~$n$.
In fact,
\begin{align*}&[\![K_{i+1}]\!]-n[\![K_i]\!]\\
=&(k_{i+1}[1,p_{i+1}-2,p_{i+1}-2,1]\pm O_{i+1})-n(k_i[1,p_i-2,p_i-2,1]\pm O_i)\\
=&k_{i+1}[1,p_{i+1}-2,p_{i+1}-2,1]-(\mp O_{i+1}+nk_i[1,p_i-2,p_i-2,1]\pm nO_i)\\
>&k_{i+1}[1,p_{i+1}-2,p_{i+1}-2,1]-(\mp O_{i+1}+nk_i[1,p_i-2,p_i-2,1]\pm n[1,p_i-2,p_i-2,1])\\
>&0.\end{align*}

Now~$\{[\![K_i]\!]\}_{i=1}^\infty$ generates a subgroup isomorphic to~$\mathbb{Z}^\infty$ in~$\mathcal{CFK}_\mathrm{alg}$ by Lemma~\ref{last1}~\ref{domination implies independence}.\linebreak
Therefore~$\{K_i\}_{i=1}^\infty$ generates a subgroup isomorphic to~$\mathbb{Z}^\infty$ in~$\mathcal{C}$.
Any nontrivial linear\linebreak combination of this family has nonvanishing~$\varepsilon$-invariant, because it maps to a nontrivial linear combination of~$\{[\![K_i]\!]\}_{i=1}^\infty$
under the homomorphism from~$\mathcal{C}$ to~$\mathcal{CFK}_\mathrm{alg}$ that maps any knot to its~$\varepsilon$-equivalence class.
\end{proof}

A knot $K$ produced by Proposition~\ref{p-2} or Lemma~\ref{phi0Upsilon} satisfies that~$\varepsilon(K)\neq0$ and that exact one of~$\Upsilon(K)$ and~$\varphi(K)$ vanishes.
In~\cite{example2}, we will give knots with nonvanishing~$\varepsilon$-invariant such that both~$\Upsilon$ and~$\varphi$ vanish.

\end{document}